\documentclass[11pt]{article}
\usepackage{amsmath}
\usepackage{amsthm}
\usepackage{amsfonts}
\usepackage{amssymb}
\usepackage[margin=1.5in]{geometry}
\theoremstyle{definition}
\usepackage{graphicx}

\usepackage{mathtools}
\mathtoolsset{showonlyrefs}
\newtheorem{theorem}{Theorem}[section]
\newtheorem{corollary}[theorem]{Corollary}
\newtheorem{definition}[theorem]{Definition}

\newtheorem{lemma}[theorem]{Lemma}
\newtheorem{proposition}[theorem]{Proposition}
\newtheorem{remark}[theorem]{Remark}

\usepackage{enumitem}
\numberwithin{equation}{section}
\usepackage{color}
\usepackage{wasysym}
\usepackage{morefloats}
\usepackage{natbib}

\newcommand{\sign}{\text{sign}}
\newcommand{\diag}{\text{diag}}

\title{The Smirnov property for weighted Lebesgue spaces }

\author{Eberhard Mayerhofer\thanks{University of Limerick, Department of Mathematics and Statistics, Castletroy, Co. Limerick, Ireland, email {eberhard.mayerhofer@ul.ie}.}}

\begin{document}
\maketitle
\thispagestyle{empty}

\vspace{-1cm}
\begin{abstract}

We establish lower norm bounds for multivariate functions within weighted Lebesgue spaces, characterized by a summation of functions whose components solve a system of nonlinear integral equations. This problem originates in portfolio selection theory, where these equations allow to identify mean-variance optimal portfolios, composed of standard European Options on several underlying assets. We elaborate on the Smirnov property—an integrability condition for the weights that guarantees the uniqueness of solutions to the system. Sufficient conditions on weights to satisfy this property are provided, and counterexamples are constructed, where either the Smirnov property does not hold, or the uniqueness of solutions fails. 
\end{abstract}

\smallskip


\maketitle
\newpage
\tableofcontents

\section{Introduction}
This paper determines sharp lower norm bounds for functions in Lebesgue spaces $\mathcal L^p(w)$ (functions of $n$ variables, weighted by a density $w$), based on their one-dimensional marginals. This problem can be framed as minimizing the $p$-norm of functions with fixed marginals. With the method of Lagrange multipliers we may reformulate this optimisation problem in terms of a system of (non-)linear integral equations subject to $n$ marginal constraints, whose solution identifies the minimizer. The problem has ties with different fields of research:

{\bf Convexity in Banach spaces:} The spaces considered in the present paper are weighted $\mathcal L^p$ spaces for $p>1$ and thus are strictly convex and reflexive Banach spaces, which implies that not only do closed subsets have elements of minimal norm \cite[Corollary 5.1.19]{megginson2012introduction} as in Hilbert spaces \cite[Theorem 4.10]{Rudin}, but this minimality can also be characterized by extending the familiar notion of orthogonality in Hilbert spaces, where the smallest element of a non-empty closed and convex set is the projection of $0$ onto the set.\footnote{As marginal constraints do indeed define closed subspaces, we may employ the results of \cite{shapiro2006topics} in this direction.}In the context to the present paper, (simplified to a bivariate setting $n=2$) we minimize the $p$-norm
\begin{equation}\label{eq: opt1}
\int_{\mathbb R^2} \vert f(x,y)\vert^p w(x,y)dx dy
\end{equation}
over all functions $f$ with marginals $g_X$ and $g_Y$, that is
\begin{equation}\label{eq: const1 intro}
\int f(x,y)w(x,y)dy=g_X(x),\quad \int f(x,y)w(x,y)dx=g_Y(y),\quad x,y\in\mathbb R.
\end{equation}

{\bf Method of Lagrange multipliers:} Since the underlying $\mathcal L^p$ space is generally not finite-dimensional\footnote{Only for discrete weights is the space countable, or it may even be isomorphic to finite-dimensional Euclidean space.}, and since marginal constraints typically introduce a continuum of constraints, the method of Lagrange multipliers used in this paper pertains to an infinite-dimensional variant, such as those discussed in \cite[Chapter 5]{ekeland} and \cite[Chapter 3]{bonnans}. Although the method of Lagrange multipliers is applied heuristically in this paper to derive the specific form of integral equations with constraints, it is demonstrated that these equations allow to solve the original optimization problem.\footnote{In other words, we rigorously verify the sufficiency of the first-order conditions for optimality within the context of a specific application, which is more crucial than proving the absolute rigor of the heuristic arguments that led to these conditions.}

{\bf Dependence Modelling:} Assuming that weights integrate to one and marginal constraints are probability densities, the problem can be interpreted as selecting the optimal dependence structure, that is a multi-variate density, from a variety of available choices. This problem is prevalent in numerous applied research fields. In finance, for instance, selecting an appropriate dependence structure is crucial for risk management and portfolio optimization. Different approaches have been used to model such dependences, e.g., Copulas, cf. \cite{nelsen,embrechts,cherubini,patton}. While almost any dependence structure is given by copulas (Sklar's theorem, \cite[Chapter 3]{nelsen}), the solution of this paper is strikingly different from any standard copula choice: For example, for bivariate problems in $\mathcal L^2$ (that is, $p=2$ and $n=2$) this paper establishes that the minimum density with marginals $g_X(x)$, $g_Y(y)$ is of the form 
\begin{equation}\label{eq: xx}
\frac{F(x)+G(y)}{2},
\end{equation}
where the functions $F,G$ satisfy the equations
\begin{align}\label{eq1xx}
&F(x) w_X(x)+\int G(y)  w(x,y)dy=2g_X(x),\\\label{eq2xx}
&G(y) w_Y(y)+\int F(x) w(x,y)dx=2g_Y(y),\\\label{eq3xx}
&\int G(y)w_Y(y)dy=0.
\end{align}
The closest construction to the functional form \eqref{eq: xx} are Archimedean copulas (cf. \cite[Chapter 4]{nelsen}), where $f=g$ belongs to  parameterized family of functions, but the inverse of $f$ is applied to the sum $f(x)+f(y)$.

The difference between our weights and copulas is threefold: they address a global optimization problem (in contrast to the popular M- or W-copulas that realize the well-known Fr\'echet-Hoeffding bounds\footnote{The lower and upper Fr\'echet-Hoeffding bounds are themselves copulas, cf. \cite[(2.2.5)]{nelsen}, known as the $W$ copula $\max(0,u+v-1)$ resp. $M$-copula $\min(u,v)$. These are, of course pointwise optimisations, unlike the ones considered in the paper.}); their structure depends on marginals — which goes against the core concept of copulas — and they can become negative under certain marginal assumptions.\footnote{The implications of potential lack of positivity for the case $p=2$ are discussed in \cite[Section EC 5.2]{gm}.}

{\bf Optimal Transport} Global optimization problems involving marginal constraints, like ours, are well-known. For example, in \cite{aglio}, instead of optimizing \eqref{eq: opt1}, the goal is to minimize
\begin{equation}\label{eq: opt2}
\int_{\mathbb{R}^2} \vert x - y \vert^p f(x,y) \, dx \, dy
\end{equation}
over all functions \( f \) that satisfy the marginal constraints \eqref{eq: const1 intro} for a unit weight \( w \). While this problem is discussed in the copula literature \cite[Exercise 6.5]{nelsen}, it is more reminiscent of the classical transport problem \cite{cedricvillani}, where the objective is to find the optimal transport plan \( f(x,y) \) that minimizes the transport cost while respecting the given marginals. The key difference between our objective \eqref{eq: opt1} and those in the transport literature suggests that a direct application of optimal transport theory is not straightforward. Moreover, although the theory of optimal transport is an active area of research in Mathematical Finance, especially with its extension to Martingale transport (cf. \cite{beiglboeck, dolinsky, backhoff}), the financial interpretation of a transport plan in our context remains unclear.

{\bf Theory of Integral Equations} The two linear integral equations \eqref{eq1xx}--\eqref{eq2xx} exhibit similarities\footnote{It is important to note that linearity in our paper applies only for \( p=2 \), and to formally align our problem with Fredholm equations, we assume that the weight \( w \) is supported on the interval \([a,b]\).} with (vector-valued) Fredholm integral equations of the first kind, where one seeks to solve for a 2-vector:
\[
\phi(u) = \left(F(u), G(u)\right)^\top
\]
in the context of the equation:
\[
g(u) = \int_a^b K(u,t) \phi(t) \, dt.
\]
Here, we define \( g(u) = (2g_X(x),2g_Y(u))^\top \), and \( K \) represents the matrix kernel\footnote{The Kronecker delta \( \delta_u \) indicates the point mass at \( u \in \mathbb{R} \).}
\[
K(u,t) := \left(\begin{array}{cc} \delta_u(t) w_X(u) & w(u,t) \\ w(t,u) & \delta_u(t) w_Y(u) \end{array}\right).
\]
While it may seem somewhat artificial to express the function \( F(x) + G(y) \) as a two-variable function \( \phi(t) \) with a scalar variable \( t \), this formulation is ill-posed, which is characteristic of Fredholm integral equations of the first kind (cf. \cite{rainerkress,tikhonov}). (We add constraint \eqref{eq3xx} to allow for unique solutions.)

%
%

%
%

{\bf Options Trading.} This study is the third in a series of papers addressing similar constrained optimization problems: Specifically, \cite{gm} addresses the Hilbertian case ($p=2$), which corresponds to a dual problem in finance: maximizing the Sharpe ratio of portfolios in markets where European options on multiple, potentially correlated, underlying assets are traded. In this context, the weight $w$ represents the joint density of $n$ risky assets, with Option contracts written on various strikes. The result shows that the solution is linear, meaning optimal payoffs are achieved by trading individual option contracts on each underlying asset rather than using basket-type options. The second paper \cite{gmz} provides a solution of the minimization problem on the hypercube for any $1<p<\infty$, where the weight is the Lebesgue measure (i.e., $w$ is the uniform density). In the special case where $p=2$, explicit expressions for the minimizers are available. These expressions imply that any square-integrable function $g:\mathbb{R}^n \to \mathbb{R}$, which integrates to one, satisfies the bound
\[
\int g^2(\xi) \, d\xi \geq \left(\sum_{i=1}^n g_i(\xi_i)^2 \, d\xi_i\right) - (n-1),
\]
where $\xi = (\xi_1, \dots, \xi_i, \dots, \xi_n)$ and the marginal $g_i$ represents the integral of $g$ with respect to all arguments except the $i$-th one, $1 \leq i \leq n$. 

This paper extends the work by providing a comprehensive analysis of the uniqueness of the involved integral equations. Specifically, \cite[Theorem 1]{gm} neither states nor proves that the equations have a unique solution, whereas uniqueness is only established in hypercubes \cite[Theorem 2.5]{gmz}. Here, we demonstrate that uniqueness is closely related to the weight satisfying the so-called {\it Smirnov property} for which the following question has a positive answer (formulated here, for simplicity, in $n=2$ dimensions):
\begin{center}
 If $f(x), g(y) \in \mathcal L^1(w)$, such that $f(x) + g(y) \in \mathcal L^q(w)$, then is $f, g \in \mathcal L^q(w)$?
\end{center}

The integral equations addressed in this paper have not been extensively studied in the literature, because they pertain to a notoriously difficult, high-dimensional problem of trading optimally many options of several underlying assets. The preprint \cite{malamud} might be the most closely related work, but it pursues different objectives and, as of now, lacks the necessary mathematical foundations. On the other hand, the uni-variate case, where a continuum of options is traded but on a \emph{single} underlying, is well understood (cf. \cite{cm} and the references cited therein), and from a mathematical perspective less demanding, as identifying a minimal discount factor is trivial in complete markets (\cite{bl}), which support only a single stochastic discount factor.\footnote{Nevertheless, the research on portfolio selection involving options with a single underlying asset has a long and rich history, most papers only selecting from very few strikes, so that they are mathematically not related to the present paper. For an overview of the literature, see \cite[Table 1.1]{gmz}.}

\subsection{Program of Paper}
In Section \ref{sec: framework}, we provide a heuristic derivation of the integral equations with constraints and develop the mathematical tools essential to the paper. This includes an exploration of "orthogonality" in weighted $\mathcal L^p$ spaces $(p>1)$, which, despite not being Hilbert spaces, are strictly convex and thereby support a form of "orthogonality" analogous to classical orthogonality in inner product spaces.

Section \ref{sec: Smirnov} is focused on the Smirnov property, detailing sufficient conditions for weights to satisfy this property and offering a counterexample where these conditions are not met, leading to the property's failure.

Section \ref{sec: mth} presents the main theorem, which identifies the element in $\mathcal L^p$ with the minimal norm that satisfies the given constraints, establishing it as the unique solution to the integral equations. A counterexample, discussed in Section \ref{sec nonuniq}, emphasizes the importance of the integrability of certain likelihood ratios, a condition consistently applied throughout the paper.

The final section concludes the paper and suggests avenues for future research.

\section{Mathematical Framework}\label{sec: framework}
\subsection{Notation}

Let $w(\xi)$ be a strictly positive density\footnote{That is, $w$ is a Lebesgue-measurable function integrating to one.} on a set $U \subset \mathbb{R}^n$, where $U$ is a Cartesian product of the form $U =  I_1\times I_2\times \dots I_n$, with each interval $I_i$ $(1\leq i\leq n)$ being a closed interval of the form $(-\infty, b]$, $[a, \infty)$, where $a, b \in \mathbb{R}$, or $[a, b]$ with $a < b$.

For $p > 1$, $\mathcal{L}^p(w)$ denotes the weighted space consisting of equivalence classes $[f]$ of Lebesgue measurable functions $f: U \to \mathbb{R}$ that satisfy
\[
\|f\|_{p} := \left( \int_U |f(\xi)|^p w(\xi) \, d\xi \right)^{1/p} < \infty.
\]
$\mathcal{L}^\infty$ represents the space of equivalence classes of real-valued essentially bounded functions on $U$.

Additionally, we use $\xi_i^c$ to denote the $(n-1)$-dimensional vector obtained by omitting the $i$-th coordinate from $\xi = (\xi_1, \dots, \xi_n)$. The marginal weight $w_i^c$ is defined as the weight $w$ integrated over the $i$-th coordinate, i.e., $w_i^c(\xi_i^c) = \int_{\mathbb{R}} w(\xi) \, d\xi_i$. Similarly, $w_i$ is the $i$-th one-dimensional marginal density, defined as $w_i(\xi_i) = \int_{\mathbb{R}^{n-1}} w(\xi) \, d\xi_i^c$.

\subsection{Heuristic Derivation of Integral Equations}

To minimize the $\mathcal{L}^p(w)$-norm $\|h\|_p$ subject to the marginal constraints
\begin{equation}\label{eq: constraints marginals p}
\int h(\xi) w(\xi) \, d\xi_i^c = g_i(\xi_i), \quad 1 \leq i \leq n,
\end{equation}
we adapt the heuristic approach used in \cite{gmz}, which addresses the problem on the hypercube without weights. Consider the Lagrangian
\[
L = \frac{1}{p} \int |h(\xi)|^p w(\xi) \, d\xi - \frac{1}{n} \sum_{i=1}^n \int \Phi(\xi_i) \left(\int h(\xi) w(\xi) \, d\xi_i^c - g_i(\xi_i) \right) d\xi_i.
\]
Setting the directional derivatives equal to zero yields the first-order conditions
\begin{equation}\label{eq: realize}
\text{sign}(h(\xi)) |h(\xi)|^{p-1} = \frac{1}{n} \sum_{i=1}^n \Phi_i(\xi_i),
\end{equation}
from which it follows that
\begin{equation}\label{eq: sum rel}
h(\xi) = \text{sign} \left( \sum_{i=1}^n \Phi_i(\xi_i) \right) \left| \frac{1}{n} \sum_{i=1}^n \Phi_i(\xi_i) \right|^{\frac{1}{p-1}}.
\end{equation}
The marginal constraints imply
\begin{equation}\label{eq: constraints marginals p}
\int \text{sign} \left( \sum_{j=1}^n \Phi_j(\xi_j) \right) \left| \frac{1}{n} \sum_{j=1}^n \Phi_j(\xi_j) \right|^{\frac{1}{p-1}} w(\xi) \, d\xi_i^c = g_i(\xi_i), \quad 1 \leq i \leq n.
\end{equation}
To uniquely determine the Lagrange multipliers $\Phi_i$—which are otherwise determined up to an additive constant—it is sufficient to impose the conditions
\begin{equation}\label{eq: super constraint}
\int \Phi_i(\xi_i) w_i(\xi_i) \, d\xi_i = 0, \quad 2 \leq i \leq n.
\end{equation}
Note that these conditions are required only for $i \geq 2$. (For the proof of uniqueness, see the end of the proof of Theorem \ref{th: main4}.)

\subsection{Orthogonality in weighted $\mathcal L^p$-spaces}
Minimality in $\mathcal L^p$-spaces \citep[Theorem 4.21]{shapiro2006topics} is characterized as follows:
\begin{lemma}\label{lemx}
Let $1<p<\infty$, $g\in\mathcal L^p(w)$, and $Y$ be a closed subspace of $\mathcal L^p(w)$. The following are equivalent:
\begin{enumerate}
\item [(i)] $\|g\|_p\leq \|g+k\|_p$ for all $k\in Y$.
\item [(ii)] $\int\sign(g(\xi)) \vert g(\xi)\vert^{p-1}k(\xi) w(\xi) d\xi=0$ for all $k\in Y$.
\end{enumerate}
\end{lemma}
A function $g$ is said to be \emph{orthogonal} to a subspace $Y$ if it satisfies any of the equivalent statements of Lemma \ref{lemx}. For $p=2$, $\mathcal L^p(w)$
is a Hilbert space, and the notion agrees with the usual orthogonality, as then (ii) of Lemma \ref{lemx} simplifies to the property of vanishing
inner product,
\[
\int g(\xi) k(\xi) w(\xi)d\xi=0,\quad k\in Y.
\]
Also, $\sign(g) \vert g\vert^{p-1}\in \mathcal L^q(w)$, where $q=p/(p-1)$ is the conjugate exponent to $p$, whence the pairing in (ii) is well defined.

\begin{lemma}\label{lem: lem2}
Let $1<p<\infty$, $f\in\mathcal L^q(w)$, where $q=p/(p-1)$, and denote by
\begin{equation}\label{eq: super}
\mathcal N:=
\left\{\phi\in \mathcal L^p(w)\Big| \int \phi(\xi)w(\xi)d\xi_i^c\equiv 0, \text{ for } 1\le i\le n\right\}.
\end{equation}
Suppose 
\begin{equation}\label{eq: bound ass}
\frac{w_i w_i^c}{w}\in\mathcal L^p(w),\quad 1\leq i\leq n.
\end{equation}
Then the following hold:
\begin{enumerate}
\item \label{extension} $\int f(\xi) w_i^c(\xi_i^c)d\xi_i^c\in\mathcal L^1(w)$ for any $1\leq i\leq n$.
\item \label{item: 0x} $\mathcal N$ is a closed subspace of $\mathcal L^q$.
\item \label{item: 1x} For any $\psi\in\mathcal L^\infty$, the function
\begin{equation}\label{eq: tilde psix}
\widetilde\psi(\xi):=\psi(\xi)-\sum_{i=1}^n\frac{w_i^c(\xi_i^c)}{w(\xi)}\int \psi(\xi)w(\xi)d\xi_i^c+(n-1)\int \psi(\eta)w(\eta)d\eta
\end{equation}
is an element of $\mathcal N$.
\item \label{eq: part 1xx} If $\int f(x)w(x)\widetilde\psi(x)dx=0$ for all $\psi\in\mathcal L^\infty$, then $f(x)=\frac{1}{n}\sum_{i=1}^n \Psi_i(x_i)$, for some functions $\Psi_i\in\mathcal L^1(w)$, $1\leq i\leq n$.
\item \label{eq: part 2xx} If $f(\xi)=\frac{1}{n}\sum_{i=1}^n \Psi_i(x_i)$, where for any $1\leq i\leq n$, $\Psi_i\in \mathcal L^q(w)$, then $\int f(x)\phi(x)w(x)dx=0$ for all $\phi\in\mathcal N$.
\end{enumerate}
\end{lemma}
\begin{proof}
The proof of \eqref{extension} is an application of Jensen's and H\"older's inequality, using \eqref{eq: bound ass}:
\begin{align*}
&\|\int f(\xi) w_i^c(\xi_i^c)d\xi_i^c\|_{w,1}=\int \left\vert \int f(\xi) w_i^c(\xi_i^c)d\xi_i^c\right\vert w_i(\xi_i)d\xi_i\\
&\qquad\leq \int \vert f(\xi)\vert w_i^c(\xi_i^c)d\xi_i^c w_i(\xi_i)d\xi_i=\int \vert f(\xi)\vert \left(\frac{w_i^c(\xi_i^c) w_i(\xi_i)}{w(\xi)}\right)w(\xi)d\xi\\
&\qquad\leq \|f\|_q\left\|\frac{w_i w_i^c}{w}\right\|_p.
\end{align*}
The proof of \eqref{item: 0x} is similar to the proof that $\mathcal M$ is closed in the proof of Theorem \ref{th: main4}.
Proof of \eqref{item: 1x}: Inspecting the sum  on the right side of \eqref{eq: tilde psix}, the first summand is, by assumption in $\mathcal L^\infty\subset\mathcal L^p(w)$, and also the last summand is in $\mathcal L^p(w)$, as it is constant. Furthermore, for any  $1\leq i\leq n$, $\frac{w_i^c(\xi_i^c)}{w(\xi)}\int \psi(\xi)w(\xi)d\xi_i^c\in\mathcal L^p(w)$, due to Jensen's inequality and \eqref{eq: bound ass}:
\begin{align*}
&\int\left|\frac{w_i^c(\xi_i^c)}{w(\xi)}\int \psi(\xi)w(\xi)d\xi_i^c\right|^p w(\xi)d\xi\leq \|\psi\|_\infty \left\|\frac{w_i w_i^c}{w}\right\|^p_p<\infty.
\end{align*}
Combining all these observations, we may conclude that $\widetilde \psi\in\mathcal L^p$. As the marginal constraints in the definition of $\mathcal N$ are fulfilled, by construction, we conclude
that $\widetilde\psi\in\mathcal N$.

To show \eqref{eq: part 1xx}, let $\psi\in\mathcal L^\infty$
such that, as proved above, $\widetilde\psi\in \mathcal N$. Fubini's theorem yields
\[
\int \left(f(x)-\sum_{i=1}^n \int f(\xi)w_i^c(\xi_i^c)d\xi_i^c+(n-1)\int f(\xi)d\xi\right)\psi(\xi) w(\xi)dx=0
\]
and since $\mathcal L^1(w)$ is dual to $\mathcal L^\infty$, we have
\[
f(\xi)=\sum_{i=1}^n \int f(\xi)w_i^c(\xi_i^c)d\xi_i^c-(n-1)\int f(\xi)d\xi
\qquad w(\xi)-\text{a.e.}.
\]
By \eqref{extension} the functions $\Psi_i(\xi_i):=n\int f(\xi)w_i^c(\xi_i^c)d\xi_i^c-(n-1)\int f(\xi)d\xi$, $1\leq i\leq n$, are in $\mathcal L^1$, and their average equals $f$, as claimed.

The proof of \eqref{eq: part 2xx} is straightforward, once one has recognised that, due to H\"older's inequality, the pairing of $\Psi_i$ and $\phi$ is well-defined, for $1\leq i\leq n$.
\end{proof}

Since $\mathcal N$ is closed, the previous two Lemmas combine to the following:
\begin{corollary}\label{cor: 1}
Let $f\in\mathcal L^q(w)$, and $\mathcal N\subset\mathcal L^p(w)$ as defined in \eqref{eq: super}. The following are equivalent:
\begin{enumerate}
\item [(i)]\label{eq: part 0x} $\|\text{sign}(f)\vert f\vert^{1/(p-1)}+\phi\|_p\geq\|f\|_q$ for all $\phi\in\mathcal N$.
\item [(ii)] \label{eq: part 1x} $\int f(x)\phi(x)w(x)dx=0$ for all $\phi\in\mathcal N$.
\end{enumerate}
Suppose, in addition, $w$ satisfies \eqref{eq: bound ass}. Then any of the two statements (i) or (ii) imply that
\begin{enumerate}
\item [(iii)]\label{eq: part 2x} $f(x)=\frac{1}{n}\sum_{i=1}^n \Psi_i(x_i)$, where $\Psi_i$ (each depending only on a single argument $x_i$) lie in $\mathcal L^1(w)$, $1\leq i\leq n$.
\end{enumerate}
Conversely, if (iii) holds with $\Psi_i\in\mathcal L^q(w)$ for $1\leq i\leq n$, then also any of the equivalent statements (i) or (ii) hold.	
\end{corollary}
\subsection{The Smirnov property}\label{sec: Smirnov}
One may wonder, whether subject to mild modifications,  (i), (ii) \it{and} (iii) can be combined into a full equivalence (such that
(i) or (ii) imply (iii) with $\mathcal L^q$ summands $\Psi_i$, $1\leq i\leq n$). We elaborate on this non-trivial issue
in the present section. To this end, we introduce the following property:
\begin{definition}\label{def: Smirnov}
Let $q>1$. A density $w$ is said to satisfy the Smirnov\footnote{This property is called after Alexander G.~Smirnov (Lebedev Physical Institute, Moscow) who pointed out that for $q=2$, any mixture density $w$ satisfies it (see also Section \ref{sec: suff} and Remark \ref{rem: smirnov}.)}
 property, if for any $f(x)=\frac{1}{n}\sum_{i=1}^n \Psi_i(x_i)\in\mathcal L^q(w)$, where $\Psi_i$ lie in $\mathcal L^1$ for $1\leq i\leq n$, we have that $\Psi_i\in\mathcal L^q(w)$ for $1\leq i\leq n$.
\end{definition}
By Corollary \ref{cor: 1} we have:
\begin{corollary}
Let $q>1$, and $f\in\mathcal L^q$. If $w$ satisfies the Smirnov property, the following are 
equivalent:
\begin{enumerate}
\item [(i)] $f(x)=\frac{1}{n}\sum_{i=1}^n \Psi_i(x_i)$, where $\Psi_i\in\mathcal L^q$, $1\leq i\leq n$.
\item [(ii)] $\int f(x)\phi(x)w(x)dx=0$ for all $\phi\in\mathcal N$.
\end{enumerate}
\end{corollary}
\subsection{Sufficient Conditions}\label{sec: suff}
The Smirnov property holds, if the density $w$ is the finite sum of product densities, each depending on a single variable only.
\begin{proposition}\label{lem: prod}
A density of the form $w=\sum_{j=1}^d \prod_{i=1}^n w^{(j)}_i(\xi_i)$, where $w^{(j)}_i\geq 0$ for
any $1\leq i\leq n$ and $1\leq j\leq d$, satisfies the Smirnov property.
\end{proposition}
\begin{proof}
Let $f(x)=\frac{1}{n}\sum_{i=1}^n \Psi_i(x_i)\in\mathcal L^q(w)$, where $\Psi_i\in\mathcal L^1(w)$, $1\leq i\leq n$.

Due to linearity, it suffices to show that $\Psi_i\in\mathcal L^q(w)$ relative to a product weight, that is $d=1$ and therefore $w=w_1(\xi_1)\cdot\dots\cdot w_n(\xi_n)$.  Furthermore, without loss of generality, we may assume that each $w_i$ integrates to one, $1\leq i\leq n$. These assumptions imply that  $w(\xi)=w_i(\xi_i) w_i^c(\xi_i^c)$ for any $1\leq i\leq n$.

As for any $1\leq i,j\leq n$, where $i\neq j$, $\Psi_j\in\mathcal L^1(w_i^c)$, It follows that $\Psi_i=n\int f(\xi)w_i^c(\xi_i^c)d\xi_i^c-\sum_{j\neq i}\int \Psi_j(\xi_j) w_j(\xi_j)d\xi_j$. Therefore, to establish the claim it suffices to show
that $\int f(\xi)w_i^c(\xi_i^c)d\xi_i^c\in\mathcal L^q(w)$.

By Jensen's inequality,
\[
\left| \int f(\xi)w_i^c(\xi_i^c)d\xi_i^c\right|^q\leq \int \vert f(\xi)\vert^q w_i^c(\xi_i^c)d\xi_i^c.
\]
Multiplying by $w$ and integrating all variables out, we get
\begin{align*}
\left\| \int f(\xi)w_i^c(\xi_i^c)d\xi_i^c\right\|^q_q&\leq \int \vert f(\xi)\vert^q w_i(\xi_i)w_i^c(\xi_i^c)d\xi\\&=\int \vert f(\xi)\vert^q w(\xi)d\xi=\|f\|_q^q<\infty,
\end{align*}
where the last inequality is by assumption, and the last identity is due to $w$ being a product of one-dimensional marginals,
that is, $w_i^c=\prod_{j\neq i}w_j$.
\end{proof}
\begin{remark}\label{rem: smirnov}
For $q=2$, Proposition \ref{lem: prod} allows a more instructive proof\footnote{I thank Alexander G Smirnov for pointing out this alternative proof.}. Assume, for simplicity, $n=2$ and
$w(x,y)=w_X(x) w_Y(y)$ (the general case is proved similarly). If $f(x), g(y)\in \mathcal L^1(w)$, and $f(x)+g(y)\in\mathcal L^2(w)$, then
\begin{align*}
&\int (f(x)+g(y))^2 w(x,y)dxdy-\int f(x) w_X(x)dx\cdot \int g(y) w_Y(y)dy\\&\qquad=\int \vert f(x)\vert^2 w_X(x)dx+\int \vert f(x)\vert^2 w_X(x)dx,
\end{align*}
and since the left side is finite, also each non-negative summand on the right one is.
\end{remark}
\begin{remark}\label{rem: discrete}
An example of practical nature involves discrete densities. Indeed, if one aims to solve equations \eqref{eq: constraints marginals p}--\eqref{eq: super constraint}, one typically discretises the weight, e.g., by setting the weights piecewise constant on a rectangular grid. (For simplicity, we use $n=2$, $U=[0,1)\times[0,1)$ and an equidistant grid of mesh-size $h=1/N$.) As the discretised $w$ can be written as
\[
w(x,y)=\sum_{i=1}^{N}\sum_{j=1}^N 1_{[(i-1)h, ih)}(x) 1_{[(j-1) h,jh)}(y),
\]
it also is of the form of Proposition \ref{lem: prod}. For discrete densities, the equations \eqref{eq: constraints marginals p}--\eqref{eq: super constraint} constitute a finite-dimensional system of non-linear equations.
\end{remark}

Another situation, where the Smirnov property holds, is characterized by essentially bounded likelihood ratios:
\begin{proposition}\label{lem: ess inf}
If 
\begin{equation}\label{eq: bound inf}
w_i w_i^c/w\in \mathcal L^\infty,\quad 1\leq i\leq n,
\end{equation}
then $w$ satisfies the Smirnov property.
\end{proposition}
\begin{proof}
Let $f(x)=\frac{1}{n}\sum_{i=1}^n \Psi_i(x_i)\in\mathcal L^q(w)$, where $\Psi_i\in\mathcal L^1(w)$, $1\leq i\leq n$. As for the proof of the previous Proposition, we only need to establish that $\int f(\xi)w_i^c(\xi_i^c)d\xi_i^c\in\mathcal L^q(w)$ for $1\leq i\leq n$. By Jensen's inequality,
\[
\left| \int f(\xi)w_i^c(\xi_i^c)d\xi_i^c\right|^q\leq \int \vert f(\xi)\vert^q w_i^c(\xi_i^c)d\xi_i^c.
\]
By assumption, there exists a positive constant $C$ such that $w_i w_i^c\leq C w$,
almost everywhere. Multiplying by $w$ and integrating all variables out, we get
\begin{align*}
\left\| \int f(\xi)w_i^c(\xi_i^c)d\xi_i^c\right\|^q_q&\leq \int \vert f(\xi)\vert^q w_i(\xi_i)w_i^c(\xi_i^c)d\xi\\&\leq C \int \vert f(\xi)\vert^q w(\xi)d\xi=\|f\|_q^q<\infty.
\end{align*}
\end{proof}
\begin{remark}\label{rem: x}
A bivariate standard normal density $w(x,y)$ with non-zero correlation has an unbounded likelihood ratio $w_X w_Y/w(x,y)$, thus does not satisfy condition \eqref{eq: bound inf}. Furthermore, this $w$ is also not of the form of Proposition \ref{lem: prod}, whence it is not clear whether $w$ satisfies the Smirnov property.
\end{remark}

\subsection{A Counterexample}\label{sec: counter}
There are densities $w$ which do not satisfy the Smirnov property. It suffices to demonstrate this in dimension $n=2$, using the domain $U=\mathbb R^2$. The following example is constructed in such a way that it violates any of the sufficient conditions formulated in the previous section to guarantee the Smirnov property (cf.~Remark \ref{rem: x} below). First,  $w_1 w_2/w\not\in\mathcal L^\infty$, thus  Proposition \ref{lem: ess inf},  does not apply. Second,  $w$ is not the finite sum of product densities (cf. Proposition \ref{lem: prod}, which demonstrates that the Smirnov robust is not robust under taking limits.

Let $q>1$ and $w_0:\mathbb R^2\rightarrow \mathbb R$ be a strictly positive ``background" density, and two functions $f(x)$, $g(y)$ that are piecewise constant on the sets $[i,i+1)$, where $i\geq 1$, satisfying further $f,g\in\mathcal L^q(w_0)$, whence $f,g\in\mathcal L^1(w_0)$. For the functions' values, we use the notation $f_i:=f(i)$ and $g_i:=g(i)$.

Let $(\theta_i)_{i=1}^\infty$ be a sequence of positive numbers summing to one such that
\begin{equation}\label{eq: sur}
\sum_{i=1}^\infty\vert f_i\vert^q\theta_i=\sum_{i=1}^\infty \vert g_i\vert^q\theta_i=\infty.
\end{equation}
In addition, assume
\begin{equation}\label{eq: l1}
\sum_{i=1}^\infty \vert f_i\vert\theta_i<\infty,\quad \sum_{i=1}^\infty \vert g_i\vert\theta_i<\infty.
\end{equation}
We further assume that
\begin{equation}\label{eq: sum2}
\sum_{i=1}^n\vert f_i+g_i\vert ^q\theta_i <\infty.
\end{equation}
(This can, e.g., be achieved by setting $f_i=-g_i$ for any $i\geq 1$.)
Then for some $\alpha\in (0,1)$, the function $w$, defined by
\begin{equation}\label{eq: counter w}
w(x,y):=\alpha w_0(x,y)+(1-\alpha)\sum_{i=1}^\infty \theta_i 1_{[i,i+1)}(x)1_{[i,i+1)}(y)
\end{equation}
is a strictly positive density on $U$. By eq.~\eqref{eq: l1}, $f, g\in\mathcal L^1(w)$ and due to \eqref{eq: sum2}, 
\begin{align*}
&\int \vert f(x)+g(y)\vert^q w(x,y)dx dy\\&\qquad=\int \vert f(x)+g(y)\vert^q w_0(x,y)dx dy+\sum_{i=1}^\infty \theta_i \vert f_i+g_i\vert ^q <\infty,
\end{align*}
but due to \eqref{eq: sur}, $f$, $g\notin\mathcal L^q(w)$.
\begin{remark}\label{rem: x}
\begin{itemize}
\item This counterexample is constructed such that most of the mass of $w$ is concentrated around the diagonal, thereby mimicking strong dependence. The addition of the background density $w_0$ makes the example density strictly positive -- which is a standing assumption of the paper. The latter, in turn, is imposed to keep likelihood ratios, such as \eqref{eq: bound ass} or \eqref{eq: bound inf} well-defined. 
\item The density $w$ violates any of the sufficient conditions formulated in the previous section to guarantee the Smirnov property. First,  $w$ is not the finite sum of product densities (cf. Proposition \ref{lem: prod}), which demonstrates that the Smirnov robust is not robust under taking limits. Also, the likelihood ratio $w_1 w_2/w\not\in\mathcal L^\infty$, thus  Proposition \ref{lem: ess inf},  does not apply. 
\item The counterexample suggests to choose $f=-g$, which implies that $f(x)+g(y)$ cannot be non-negative. In financial applications, where the sum is related to the stochastic discount factor (cf. equation \eqref{eq: sum rel} above, as well as \cite[Figure EC.2 and Section EC.5.2]{gm}), negative signs lead to negative prices of certain, typically not traded, basket options. On the other hand, if $f,g\geq 0$, then such counterexample does not exist. In fact, since for any $q>1$, we have by Jensen's inequality,
\[
f^q(x)+g^q(y)\leq (f(x)+g(y))^q
\]
and thus $f(x)+g(y)\in\mathcal L^q(w)$ implies $f$, $g\in \mathcal L^q(w)$, which conflicts with assumption \eqref{eq: sur}, or \eqref{eq: sum2} cannot be satisfied.

\end{itemize}

\end{remark}

\section{Main Results}\label{sec: mth}
\subsection{Theorem  and Proof}
\begin{theorem}\label{th: main4}
Let $p>1$, and assume that $w$ satisfies \eqref{eq: bound ass}. If $g\in\mathcal L^p(w)$ is such that
$g_i\in\mathcal L^p(w)$ for $1\leq i\leq n$, then it satisfies the bound
\begin{equation}\label{eq: estimate a}
\int \vert g(\xi)\vert ^pw(\xi) d\xi\geq \int\vert\overline\Phi(\xi)\vert^{\frac{p}{p-1}}w(\xi)d\xi,
\end{equation}
where 
\[
\overline \Phi(\xi):=\frac{1}{n}\sum_{i=1}^n \Phi_i(\xi_i)
\]
and $\Phi_i$ are the solutions of the system of integral equations \eqref{eq: constraints marginals p}--\eqref{eq: super constraint}.

If $w$ satisfies the Smirnov property, then the solutions are unique and equality holds in \eqref{eq: estimate a} if and only if
\begin{equation}\label{gp}
g(\xi)=\sign\left(\overline \Phi(\xi)\right)\left\vert\overline \Phi(\xi)\right\vert^{ \frac{1}{p-1} }.
\end{equation}
\end{theorem}
\begin{proof}
Note that it is not obvious (but can be proved, under extra assumptions on $g$) that $g_i:=\int g(\xi)w(\xi)d\xi_i^c\in\mathcal L^p(w)$ for $1\leq i\leq n$, hence we have assumed it. For the proof, we follow the lines of the corresponding proof of \cite[Theorem 2.5]{gmz}, making the appropriate adaptions, especially concerning the inclusion of weights and references to the relevant adaption made in the present paper for dealing with the non-Hilbertian cases.

By assumption, the set
\[
\mathcal M:=
\left\{h\in \mathcal L^p \Big| \int h(\xi) w(\xi)d\xi_i^c= g_i(\xi_i),
\quad 1\leq i\leq n\right\}
\]
is well-defined, and it is non-empty because $g\in\mathcal M$. The set is convex, by construction. To show that it is closed, let $h_n\in\mathcal M$ and $\lim_{n\rightarrow \infty} h_n=h$ in $\mathcal L^p$. Then the sequence $(h_n)_{n\ge 1}$ is uniformly integrable, hence by Vitali's convergence theorem, $\xi_i$- almost everywhere,
\[
\int h(\xi)w(\xi)d\xi_i^c=\int \lim_{n\rightarrow\infty} h_n(\xi)w(\xi)d\xi_i^c=\lim_{n\rightarrow\infty }\int h_n(\xi) w(\xi)d\xi^c=g_i(\xi_i),
\]
which proves that $h\in \mathcal M$, whence $\mathcal M$ is a closed, convex and non-empty set. Denote by $h_*$ the unique element in $\mathcal M$ of smallest norm.\footnote{In a strictly convex and reflexive Banachspace, any
non-empty, closed convex set has an element of smallest norm, see \citep[Corollary 5.1.19]{megginson2012introduction}.} We claim that $h_*=g$, where $g$ is defined in \eqref{gp}. To this end, introduce the function space
\begin{equation}\label{eq: Ndef}
\mathcal N:=\left\{\phi\in \mathcal L^p\Big | \int \phi(\xi)w(\xi)d\xi_i^c\equiv 0, 1\le i\le n \right \},
\end{equation}
which is closed also (set $g=0$ in the definition of $\mathcal M$, in which case $\mathcal M=\mathcal N$, and use the fact that $\mathcal M$ is closed, as is proved above).
By the minimality of $h_*$, it follows that for any $\varepsilon>0$ and any $\phi\in \mathcal N$
\begin{equation}\label{eq: min}
\|h_*\pm \varepsilon \phi\|_p^p-\|h_*\|_p^p\geq 0,
\end{equation}
and therefore, by Lemma \ref{lemx},
\begin{equation}\label{eq: ortho}
\int \sign(h_*(\xi))\vert h_* (\xi)\vert^{p-1} \phi(\xi)w(\xi) d\xi=0,\quad \phi\in\mathcal N.
\end{equation}
(Note that 
\begin{equation}\label{eq: it is good}
\vert h_*\vert^{p-1}\in L^q(w), 
\end{equation}
where $q=\frac{p}{p-1}$, hence the above pairing is finite, by H\"older's inequality.) Corollary \ref{cor: 1} yields
\[
\sign(h_*(\xi))\vert h_*(\xi)\vert^{p-1}=\overline \Phi(\xi), \quad \text{where}\quad \overline\Phi(\xi):=\frac{1}{n}\sum_{i=1}^n \Phi_i(\xi_i),
\]
with measurable functions 
\begin{equation}\label{eq: it is even better}
\Phi_i(\xi_i)\in\mathcal L^1(w),\quad 1\le i\le n, 
\end{equation}
each depending on one variable $\xi_i$ only. Because $\sign(h_*)=\sign(\overline\Phi(\xi))$, it follows that
\begin{equation}\label{eq super xx}
h_*(\xi_1,\dots,\xi_n)=\sign( \overline\Phi(\xi))\left\vert\overline\Phi(\xi)\right\vert^{\frac{1}{p-1}}
\end{equation}
and $\overline \Phi$ solves the nonlinear integral equations \eqref{eq: constraints marginals p} for $1\leq i\leq n$. As these equations involve the sum $\overline \Phi$ only, we can satisfy the extra constraints \eqref{eq: super constraint}, by replacing $\Phi_i$ by $\Phi_i-\int \Phi_i(\xi_i)w_i(\xi_i)d\xi_i$ ($2\leq i\leq n$), if necessary.

It remains to show the uniqueness. Assume, in addition, that \( w \) satisfies the Smirnov property, as defined in Definition \ref{def: Smirnov}. Then, due to \eqref{eq: it is good} (which implies that \( \overline{\Phi} \in \mathcal{L}^q(w) \)) and \eqref{eq: it is even better}, we infer from the Smirnov property that \( \Phi_i \in \mathcal{L}^q(w) \) for \( 1 \leq i \leq n \). Assume that, in addition to \( \overline{\Phi} \), the function \( \overline{\Psi}(\xi) := \frac{1}{n} \sum_{i=1}^n \Psi_i(\xi_i) \) also solves \eqref{eq: constraints marginals p}--\eqref{eq: super constraint}. By Corollary \ref{cor: 1}, the function \( h := \text{sign}(\overline{\Psi}) \cdot |\overline{\Psi}|^{\frac{1}{p-1}} \) is orthogonal to \( \mathcal{N} \) defined in \eqref{eq: Ndef}. Furthermore, by \eqref{eq: constraints marginals p}, \( h - h_* \in \mathcal{N} \), hence by the definition of orthogonality, we find that \( \|h\|_p \leq \|h_*\|_p \). In view of \eqref{eq: min}, it follows that \( h = h_* \), whence also \( \overline{\Psi} = \overline{\Phi} \). As \( \overline{\Phi}(\xi) = \frac{1}{n} \sum_{i=1}^n \Phi_i(\xi_i) = \frac{1}{n} \sum_{i=1}^n \Psi_i(\xi_i) =: \overline{\Psi}(\xi) \) almost everywhere, the extra constraints given by \eqref{eq: super constraint} yield, upon integrating \( n\overline{\Phi} = n\overline{\Psi} \) with respect to \( w_i(\xi_1) d\xi_1 \), that \( \Phi_1(\xi_1) = \Psi_1(\xi_1) \) \( w_1 \)-almost everywhere (as the rest of the integrals vanish). Applying the constraint for \( i = 2 \), it follows that

\begin{align*}
\int \Phi_1(\xi_1) w_1(\xi_1) d\xi_1 + \Phi_2(\xi_2) + 0 &= \int \Psi_1(\xi_1) w_1(\xi_1) d\xi_1 + \Psi_2(\xi_2) + 0 \\&= \int \Phi_1(\xi_1) w_1(\xi_1) d\xi_1 + \Psi_2(\xi_2),
\end{align*}

from which it follows that \( \Phi_2(\xi_2) = \Psi_2(\xi_2) \) \( w_2 \)-almost everywhere. Continuing similarly for \( 3 \leq i \leq n \), it follows that \( \Phi_i = \Psi_i \) \( w_i \)-almost everywhere for \( 3 \leq i \leq n \).

\end{proof}

\subsection{A Counterexample concerning Uniqueness}\label{sec nonuniq}

Using the density $w$ in equation \eqref{eq: counter w}, we can see that, if the conditions in Theorem \ref{th: main4} are violated, uniqueness for the integral equations fails. To keep the example simple, we consider only the Hilbertian case, that is, $p=2$. 

The weight $w$ in Section \ref{sec: counter}, does not satisfy the Smirnov property. Let us use $f_i=-g_i$, for all $i\geq 1$, then $f(x)=\sum_{i=1}^\infty f_i  1_{[i,i+1)}(x)=-g(x)$ and thus $f(x)+g(y)\in\mathcal L^2(w)$. Then for appropriate choices of $f_i$, $i\geq 1$, $f,g\in\mathcal L^1(w)$, but $f,g\notin\mathcal L^2(w)$. 

Let us take the extreme case $\alpha=0$ in \eqref{eq: counter w}, which is excluded in Section \ref{sec: counter}. In this case $w$ is supported around the diagonal, and vanishes away from it. Studying uniqueness, we assume that the marginals $g_1=g_2=0$. Then the integral equations  \eqref{eq: constraints marginals p}--\eqref{eq: super constraint} are linear  
\begin{align}\label{eq1}
&f(x) w_X(x)+\int g(y)  w(x,y)dy=0,\\\label{eq2}
&g(y) w_Y(y)+\int f(x) w(x,y)dx=0,\\\label{eq3}
&\int g(y)w_Y(y)dy=0.
\end{align}

$f=g=0$ satisfy these equations. But also non-trivial solutions can be constructed, as \eqref{eq3} is easy to satisfy, for instance, if one sets $f_1:=-\frac{\sum_{i=2}^\infty f_i \theta_i}{\theta_1}$, then,
\[
\int g(y) w_Y(y)dy=-\sum_{i=1}^\infty f_i \theta_i=0.
\]
Due to symmetry of $w$ and $f=-g$, equations \eqref{eq1} and \eqref{eq2} are collinear. Furthermore, for any $k\geq 1$, and $x\in [k,k+1)$ we have $f(x)=f_k$,
and $w(x,y)=\theta_k 1_{[k,k+1)}(y)$, therefore \eqref{eq1} becomes
\[
f_k \theta_k-\int\left( \sum_{j=1}^\infty f_j 1_{[j,j+1]}(y)\right)\theta_k 1_{[k,k+1)}(y)dy=f_k\theta_k-\int_k^{k+1} f_k\theta_k dy=0.
\]

In this counterexample, the condition \eqref{eq: bound ass} is violated, as the integral $\int w_X^2 w_Y^2/w$ is infinite (because the denominator $w$ vanishes away from the ``diagonal"
\[
\bigcup_{k=1}^\infty [k,k+1)\times [k,k+1)\subset \mathbb R^2,
\]
while the product $w_X\times w_Y$ is strictly positive on $[1,\infty)\times [1,\infty)$.

\subsection{Mixture models}\label{sec: mixtures}
Mixture densities, as discussed in \cite{lachlan}, serve as a powerful tool in various applications for modeling complex dependencies by combining simpler, well-understood components. By Proposition \ref{lem: prod}, certain mixture densities also satisfy the Smirnov property, namely those who are mixtures (that is, sums) of product densities, where each factor depends on one variable only (in the following we abbreviate these as ``one-mixtures").  Even though each summand in such a mixture represents the density of $n$, independent random variables, mixing does not imply the same.  In particular, such mixing allows to model non-zero correlation. The special feature of one-mixtures in the $\mathcal L^2$ context of this paper turn the (linear) integral equations into a system of linear equations, which are particularly easy to treat. \footnote{The general case of mixture distributions, without reference to the Smirnov property, was introduced by \cite{gm}, but was not explored in depth. Since their constraints and mixture models differ slightly from ours, the linear equations also exhibit some differences. Most notably, \cite{gm} does not demonstrate that the system of \( 2n \) equations in \( 2n \) unknowns has maximal rank, and thus they do not establish the unique solvability of the system.}

As an example, we mix $k$ bivariate densities $w_X^i(x) w_Y^i(y)$. Note that, while it is unknown, whether the Smirnov property holds for the any bivariate  density, not even for normal ones (cf.~Remark \ref{rem: x}), we have this property for one-mixtures due to Proposition \ref{lem: prod}. With weights $\alpha^i\in (0,1)$, the mixture density takes the form
\[
w(x,y)=\sum_{i=1}^k \alpha^i w_X^i(x) w_Y^i(y),
\]
where $\sum_{i=1}^k \alpha^i=1$, which normalizes the weight $w$ to having unit mass. Thus the marginals of $w$ are given by
\[
w_X(x)=\sum_{i=1}^k \alpha^i w_X^i(x),\quad w_Y(y)=\sum_{i=1}^k \alpha^i w_Y^i(y).
\]

An inspection of the integral equations \eqref{eq1xx}--\eqref{eq2xx} reveals that the element of minimal norm is of the form $\frac{F(x)+G(y)}{2}$, where
\begin{equation}\label{eq: F1}
F(x)=\frac{2g_X(x)-\sum_{i=1}^k \alpha^i c_Y^i w_X^i(x)}{w_X(x)}, \quad c_Y^i:=\int_{-\infty}^\infty G(y)w_Y^i(y)dy
\end{equation}
and, quite similarly,
\begin{equation}\label{eq: G1}
G(y)=\frac{2g_Y(y)-\sum_{i=1}^k \alpha^i c_X^i w_Y^i(y)}{w_Y(y)}, \quad c_X^i:=\int_{-\infty}^\infty F(x)w_X^i(y)dx.
\end{equation}
This appears a recursive problem, but we actually have reduced the problem to finding the $2k$ constants $c_X^i$
and $c_Y^i$ ($1\leq i\leq k$): Plugging the Ansatzes for $F,G$ from the left sides of \eqref{eq: F1}--\eqref{eq: G1} back into the integral equations \eqref{eq1xx}--\eqref{eq3xx} yields the linear equations\footnote{Note that all the integrals are finite, because 
\begin{align*}
&\int_{-\infty}^\infty\frac{w_Y^i(y) w_Y^j(y)}{w_Y(y)}dy\leq\int_{-\infty}^\infty \frac{1}{2}\frac{(w_Y^i(y))^2+(w_Y^j(y))^2}{w_Y(y)}dt\\&\leq \int_{-\infty}^\infty\frac{1}{2\alpha^i}w_Y^i(y)dy+\frac{1}{2\alpha^j}w_Y^j(y)dy=\frac{1}{2\alpha^i}+\frac{1}{2\alpha^j}
\end{align*}
and
\[
\int_{-\infty}^\infty\frac{g_Y(y) w_Y^i(y)}{w_Y(y)}dy\leq \frac{1}{\alpha^i}\int_{-\infty}^\infty g_Y(y)dy=\frac{1}{\alpha^i}.
\]
(The rest of the integrals are estimated similary.)}
\begin{align}\label{eq: mix1}
c_Y^i+\sum_{j=1}^k \alpha^j c_X^j\int_{-\infty}^\infty\frac{w_Y^i(y) w_Y^j(y)}{w_Y(y)}dy&=\int_{-\infty}^\infty\frac{2g_Y(y) w_Y^i(y)}{w_Y(y)}dy,\; 1\leq i \leq k,\\\label{eq: mix2}
c_X^i+\sum_{j=1}^k \alpha^j c_Y^j \int_{-\infty}^\infty \frac{w_X^i(x) w_X^j(x)}{w_X(x)}dx&=\int_{-\infty}^\infty \frac{2g_X(x) w_X^i(x)}{w_X(x)}dx,\;1\leq i \leq k,\\\label{eq: mix3}
\sum_{i=1}^k \alpha^i c_X^i&=2.
\end{align}
These are $2k+1$ equations in $2k$ unknowns, but the first $2k$ equations are not linearly independent, because the sum of equations $1$ to $k$ is equal to the sum of equations $k+1$ to $2k$. In view of the second part of Theorem \ref{th: main4}, which guarantees uniquees of the equations, we may strike one of the first $2k$ equations to obtain a system of maximal rank.


For a concrete example, let us sample from a bivariate normal distribution with zero mean, unit variances and correlation parameter $\rho>0$, that is a normal distribution on $\mathbb R^2$ with parameters $\mu=(0,0)^\top$, and variance-covariance matrix
\[
\Sigma=\left(\begin{array}{ll} 1 & \rho\\ \rho &1\end{array}\right).
\]
We obtain parameter estimates for a normal mixture model with two components ($k=2$), each with means ($\mu_X^i$, $\mu_Y^i$) and diagonal covariances $\Sigma^i=\diag(\sigma_X^i,\sigma_Y^i)$, for $1\leq i\leq k$ (cf.~ Figure \ref{fig: 1}). 
The fact that we only know that normal mixture distributions with diagonal covariances satisfy the Smirnov property (but not of correlated ones, cf.~ Remark \ref{rem: x}) works to our advantage here, because assuming diagonal covariance avoids overfitting. We have also experimented with the number of components, only realizing overfitting occuring for three or more mixing densities, which suggested to keep $k=2$.

\begin{figure}[h]
    \centering
    \includegraphics[width=0.8\textwidth]{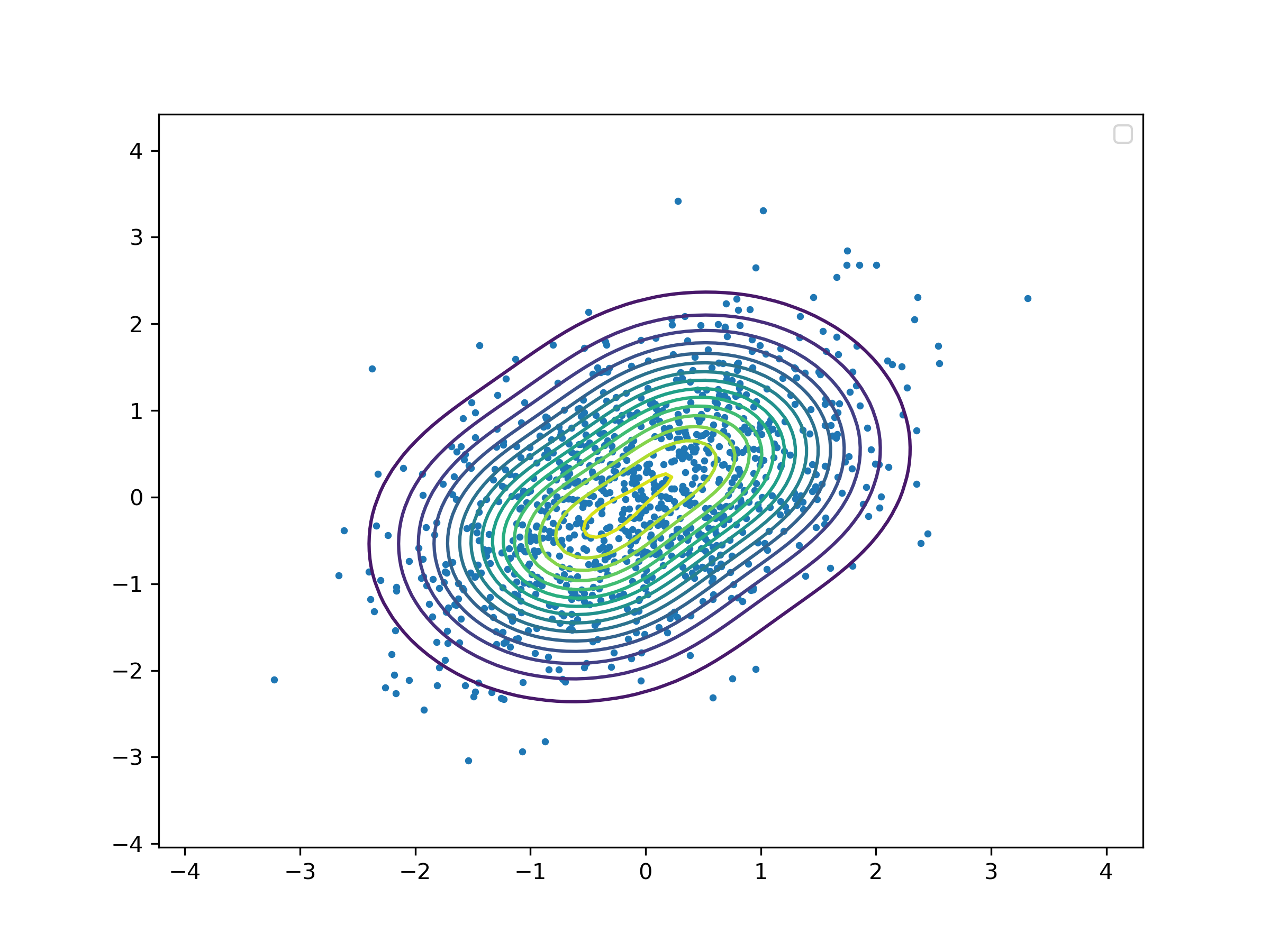}
    \caption{This 2D Contour plot depicts the fit of a Gaussian mixture model  to a sample of size $1000$ from a correlated bivariate normal distribution with zero means, unit variances and correlation $\rho=0.5$. The fitted normal mixture is comprised of two normal densities, each with zero correlation. Their estimated weights $\alpha_i$ are essentially the same, with the first one $\alpha_1=0.0.498582$. The other parameter estimates are $\mu^1=(-0.64,-0.62)^\top$, $\mu^2=(0.62,0,62)$, $\Sigma^1=\diag(0.62, 0.62)$ and $\Sigma^1=\diag(0.61, 0.64)$.}
    \label{fig: 1}
\end{figure}
The estimation of Figure \ref{fig: 1} suggests that $\alpha=1/2$, $\mu_X^1=\mu_Y^1=-a<0$, and $\mu_X^2=\mu_Y^2=a$, and $\Sigma^1=\Sigma^2=\diag(a^2,a^2)$. This implies that
\[
w_Y(y)=\frac{1}{2b\sqrt{2\pi}}e^{-\frac{(a+y)^2}{2a^2}}\left(1+e^{\frac{2y}{a^2}}\right),\quad x\in\mathbb R
\]
and, by symmetry,
\[
w_X(x)=w_Y(x), \quad x\in\mathbb R.
\]
We thus get the analytic expressions
\[
\kappa_1:= \int_{-\infty}^\infty \frac{w_Y^1(y) w_Y^2(y)}{w_Y(y)}dy= \int_{-\infty}^\infty \frac{w_X^1(x) w_X^2(x)}{w_X(x)}dx=\int_{-\infty}^\infty \sqrt{\frac{2}{\pi }}\frac{e^{-\frac{(a-x)^2}{2a^2}}dx}{a(1+e^{\frac{2x}{a}})}
\]
and
\[
\kappa_2:= \int_{-\infty}^\infty\frac{(w_Y^1(y))^2}{w_Y(y)}dy=\int_{-\infty}^\infty \frac{(w_X^1(x))^2 }{w_X(x)}dx=\int_{-\infty}^\infty\sqrt{\frac{2}{\pi }}\frac{e^{-\frac{(a+x)^2}{2a^2}}dx}{a(1+e^{\frac{2x}{a}})}.
\]
(For example, if $a=0.65$ we get $\kappa_1\approx 0.4496$ and $\kappa_2\approx 1.5504$.)
Assuming standard normal marginals $g_X,g_Y$, we further get\footnote{The result is actually exact, because the integrand, which is of the form 
\[
\frac{1}{\sqrt{\pi}}\frac{e^{(a-x)}x}{1+e^{ax}},
\]
where $a>0$, is a density itself.}
\[
\int_{-\infty}^\infty\frac{g_X(x) w_X^i(x)}{w_X(x)}dx=\int_{-\infty}^\infty\frac{g_Y(y) w_Y^i(y)}{w_Y(y)}dy=1.
\]
Thus, the five equations \eqref{eq: mix1}, \eqref{eq: mix2} and \eqref{eq: mix3} take the form
\begin{align*}
c_Y^1+\frac{\kappa_2}{2} c_X^1+\frac{\kappa_1}{2} c_X^2&=2,\\
c_Y^2+\frac{\kappa_1}{2} c_X^1+\frac{\kappa_2}{2} c_X^2&=2,\\
c_X^1+\frac{\kappa_2}{2} c_Y^1+\frac{\kappa_1}{2} c_Y^2&=2,\\
c_X^2+\frac{\kappa_1}{2} c_Y^1+\frac{\kappa_2}{2} c_Y^2&=2,\\
c_X^1+c_X^2&=4.
\end{align*}
Since $\kappa_1+\kappa_2=2$, one of the first four equations is redundant, and thus can be stricken out. The unique solution of this system is\footnote{The system is of maximal rank, as the determinant of the coefficient matrix is given by $1-(\kappa_1-\kappa_2)^2/4$, which must be non-zero because $\kappa_{1,2}\neq 0$ and $\kappa_1+\kappa_2=2$.}
\[
c_X^1=c_X^2=2,
\]
and
\[
c_Y^1=c_Y^2=0.
\]

Thus, the solution is
\[
\frac{F(x)+G(y)}{2}=\frac{g_X(x)}{w_X(x)}+\frac{g_Y(y)-\frac{w_Y^1(y)+w_Y^2(y)}{2}}{w_Y(y)}=\frac{g_X(x)}{w_X(x)}+\frac{g_Y(y)}{w_Y(y)}-1.
\]
Note that, despite the weight being a non-trivial mixture distribution, the solution
is of the same functional form as if $w$ were a product density (that is, of the form $w(x,y)=w_X(x)w_Y(y)$, where the equations \eqref{eq1xx}--\eqref{eq3xx} immediately gives that solution.) 

\section{Conclusion}
In this paper, we have examined a key characteristic of multivariate weights—the Smirnov property—which plays a crucial role in identifying sharp lower \( p \)-norm estimates for Lebesgue-measurable functions subject to specific marginal constraints \eqref{eq: constraints marginals p}. These constraints imply that minimal solutions take a specific functional form: powers of arithmetic averages of functions, each depending on a single univariate argument (see eq. \eqref{eq: sum rel} and Theorem \ref{th: main4}). This formulation enables us to rewrite the problem as a system of (non-)linear integral equations, subject to constraints \eqref{eq: constraints marginals p}--\eqref{eq: super constraint}. The unique solvability of these equations is ensured by the Smirnov property. Several important questions for future research emerge from this work:

As a consequence of Proposition \ref{lem: prod}, any weight $w$ can be approximated either by discrete distributions with compact support (see Remark \ref{rem: discrete}) or by mixture models (Section \ref{rem: discrete}), in such a way that the approximating weight satisfies the Smirnov property. However, some well-known weights, such as the bivariate normal density (cf.~Remark \ref{rem: x}), which are frequently used in modeling, do not satisfy the conditions of Proposition \ref{lem: prod} or the bound in \eqref{eq: bound inf}. Consequently, it is currently unknown whether these weights satisfy the Smirnov property or allow for the unique solvability of the integral equations. To deepen the understanding of this issue, we provide a counterexample in Section \ref{sec: counter} that violates the Smirnov property, as well as another counterexample in Section \ref{sec nonuniq} showing that uniqueness may fail if the integrability condition of Theorem \ref{th: main4} is not met.

For the Hilbertian case, where \( p = 2 \), this paper addresses issues related to identifying the minimal stochastic discount factor (SDF), a topic that has been extensively studied by \cite{gm} in the context of options portfolio selection using a mean-variance criterion. We establish that this SDF is the unique solution of the associated integral equations, provided the Smirnov property holds (thus extending the results of \cite[Theorem 1 (iii)]{gm} with Theorem \ref{th: main4}). In this framework, the marginal constraints are determined by observed option prices on a single underlying asset. Since only a finite number of options are traded in practice, one could directly model the problem using discrete distributions, which inherently satisfy the Smirnov property (see Remark \ref{rem: discrete}). Consequently, a minimal SDF can be uniquely identified in this discrete case. However, when continuous distributions are employed to model dependence structures, identifying the correct solution becomes more challenging, as it is unclear whether the Smirnov property holds. Typically, these equations are solved numerically, with each discretization yielding a unique solution. An open problem remains as to whether successive refinements of the discretization meshes could lead to a well-defined and correct solution in the limit as the mesh size tends to zero.

A comparable duality theory for investors aiming to maximize the power utility of terminal wealth leads to problems in weighted $L^p$ spaces, where $0<p<1$, and thus in non-convex Banach spaces. Consequently, the "orthogonality" Lemma \ref{lemx} does not apply in this context, leaving the analysis of this important problem entirely open. The ramifications of this will be addressed in future research.

Another area for future research involves the optimal selection of options with not only many strikes, but different maturities. Such a problem results in more complicated systems of integral equations, because not only one density needs to be fitted to marginals, but entire finite dimensional distributions. A related problem, though with less conventional objectives, was addressed by \cite{malamud}, who aimed to identify multivariate transition densities of a Markov chain.

\end{document}